\documentclass[12pt]{amsart} 
\usepackage{amscd}
\usepackage{amssymb}
\usepackage{a4wide}
\usepackage{amstext}
\usepackage{amsthm}
\usepackage{xcolor}
\usepackage{cite}
\usepackage[T1,T2A]{fontenc}
\usepackage[utf8]{inputenc}
\usepackage[american]{babel}
\usepackage{url}
\usepackage{amsfonts}
\usepackage{amssymb, amsthm}
\usepackage{amsmath}
\usepackage{mathtools}
\usepackage{needspace}
\usepackage[pdftex]{graphicx}
\usepackage{hyperref}
\usepackage{datetime}
\usepackage{epigraph}
\usepackage{verbatim}
\usepackage{mathtools}
\usepackage{xcolor}
\usepackage{enumerate}
\usepackage[american]{babel}
\numberwithin{equation}{section}

\newcommand{\N}{\mathbb{N}}
\newcommand{\R}{\mathbb{R}}

\newcommand{\Cm}{\mathbb{C}}

\newcommand{\eps}{\varepsilon}

\newcommand{\B}{\mathcal{B}}
\newcommand{\F}{\mathcal{F}}

 \DeclareMathOperator{\re}{Re}

\DeclareMathOperator{\supp}{supp} 
\DeclareMathOperator{\Id}{Id}

\renewcommand{\phi}{\varphi}

\newtheorem{Thm}{Theorem}[section]
\newtheorem{theorem}[Thm]{Theorem}

\newtheorem{lemma}[Thm]{Lemma}
\newtheorem{proposition}[Thm]{Proposition}

\begin{document}
\sloppy

\title[]{Exponential lower bound for the eigenvalues of the time-frequency localization operator before the plunge region}
\author{Aleksei Kulikov}
\address{Tel Aviv University, School of Mathematical Sciences, Tel Aviv, 69978, Israel,
\newline {\tt lyosha.kulikov@mail.ru} 
}
\begin{abstract} { We prove that the eigenvalues $\lambda_n(c)$ of the time-frequency localization operator satisfy $\lambda_n(c) > 1 -\delta^c$ for $n = [(1-\eps)c]$, where $\delta = \delta(\eps) < 1$ and $\eps > 0$ is arbitrary,  improving on the result of Bonami, Jaming and Karoui, who proved it for $\eps \ge 0.42$. The proof is based on the properties of the Bargmann transform.
}
\end{abstract}
\maketitle
\section{Introduction}
For a measurable set $\Omega \subset \R$ we define the projection $P_\Omega:L^2(\R)\to L^2(\R)$ as $P_\Omega f = f\chi_{\Omega}$ and the Fourier projection $Q_\Omega:L^2(\R)\to L^2(\R)$ as $Q_\Omega = \mathcal{F}^{-1}P_\Omega\mathcal{F}$, where $\mathcal{F}$ is the Fourier transform
$$\mathcal{F}(f)(\xi) = \int_\R f(x)e^{-2\pi i x \xi}dx.$$

For a pair of sets $T, \Omega\subset \R$ we define the time-frequency localization operator $S_{T, \Omega}$, associated with them, as $S_{T, \Omega} = P_TQ_\Omega P_T$.  It is easy to check that $S_{T, \Omega}$ is a bounded self-adjoint non-negative definite operator. 

If the measures of $T$ and $\Omega$ are finite then $S_{T, \Omega}$ is a Hilbert--Schmidt operator with the Hilbert--Schmidt norm equal to $|T||\Omega|$ (see \cite[proof of Theorem 2.3.1]{Gro}). In particular, in this case $S_{T, \Omega}$ is a compact operator and as such it has a sequence of eigenvalues $1 > \lambda_1(T, \Omega) \ge \lambda_2(T, \Omega) \ge \ldots > 0$. The first eigenvalue $\lambda_1(T, \Omega)$ is equal to the norm of $S_{T, \Omega}$ and in particular it is always at most $1$, but it can be shown that it is always strictly less than $1$ (see \cite[Theorem 2.3.3]{Gro}). A highly non-trivial result of Nazarov \cite[Theorem II]{Naz} shows that there exist absolute constants $c, C > 0$ such that we always have $\lambda_1(T, \Omega) \le 1 -ce^{-C|T||\Omega|}$ regardless of the geometry of sets $T$ and $\Omega$.

The famous Donoho--Stark conjecture \cite{Don} says that if $T$ is an interval then among the sets $\Omega$ of fixed measure the maximum of $\lambda_1(T, \Omega)$ is achieved when both $T$ and $\Omega$ are intervals. In this case, it can be seen by dilation that the eigenvalues depend only on the product of lengths of the intervals $c = |T||\Omega|$, so we have a sequence $1 > \lambda_1(c) > \lambda_2(c) > \ldots > 0$. The distribution of these eigenvalues   is the main subject of this paper.

 It turns out, as was discovered by Slepian \cite{Sle} and rigorously proved by Landau and Widom \cite{Lan}, that the eigenvalues exhibit a phase transition around the point $n_0 = c$: if $c-n \gtrsim \log c$ then $\lambda_n(c) \approx 1$, if $n - c\gtrsim \log c$ then $\lambda_n(c)\approx 0$, and only in the plunge region $|c-n|\lesssim \log c$ the eigenvalues $\lambda_n(c)$ have intermediate value. Specifically, they proved the following theorem.

\begin{theorem}\label{Slepian}
For a fixed $b\in \R$ we have 
\begin{equation}\label{Slepian eq}
\lim_{c\to \infty} \lambda_{n(c, b)}(c) = (1+e^b)^{-1},
\end{equation}
where $n(c, b) = [c + \frac{1}{\pi^2}b\log(c)]$ and $[t]$ denotes the integer part of $t$.
\end{theorem}

Their proof had no uniformity in $b$ and so in recent years there were some results on getting similar results and bounding the size of the plunge region for varying values of $b$, in particular by Israel \cite{Isr}, culminating in the following result of Karnik, Romberg and Davenport \cite{Kar} (note that its formulation is slightly different since we use a different normalization of the Fourier transform).
\begin{theorem}[{\cite[Theorem 3]{Kar}}]
For all $c > 0$ and $0 < \eps < \frac{1}{2}$ we have
\begin{equation}\label{Karnik eq}
\left|\{n: \eps < \lambda_n < 1-\eps\}\right| \le \frac{2}{\pi^2}\log(50c+25)\log\left(\frac{5}{\eps(1-\eps)}\right) + 7.
\end{equation}
\end{theorem}

For $c\to \infty, \eps \to 0$ their bound has the form $\log(c)\log(\eps^{-1})(\frac{2}{\pi^2}+o(1))$, coinciding with the bound from Theorem \ref{Slepian} asymptotically while being completely explicit and uniform in all parameters.

There is also quite a few results establishing asymptotic and non-asymptotic decay bounds on the eigenvalues $\lambda_n(c)$ when $n > c$. Widom \cite{Wid} showed that for fixed $c$ the eigenvalues decay like $\lambda_n(c) \sim \left(\frac{e\pi c}{(8n+4)}\right)^{2n+1}$, and the works of Osipov \cite{Osi}, Bonami and Karoui \cite{Bon} and Bonami, Jaming and Karoui \cite{Jam} established uniform in $c$ upper bounds on $\lambda_n(c)$, in particular \cite[Theorem 1]{Bon} essentially says that after the plunge region (say, for $n > (1+\eps)c$) the eigenvalues start with an exponential decay and then catch up to the super-exponential decay similar to the Widom's result.

On the other hand, not much is known about how close are the eigenvalues $\lambda_n(c)$ to $1$ when $n$ is significantly less than $c$. Apart from the general bounds on the plunge region like \eqref{Karnik eq}, we are aware of only two results dealing with this regime: the result of Fuchs \cite{Fuc} who showed that for fixed $n$ the eigenvalues $\lambda_n(c)$ satisfy
$$1-\lambda_n(c) \sim \frac{\sqrt{\pi}}{2}\frac{8^{n}}{(n-1)!}\left(\frac{\pi c}{2}\right)^{n-1/2}e^{-\pi c}, c\to \infty, $$
in particular showing that $\lambda_n(c)$ are exponentially close to $1$, and the following result of Bonami, Jaming and Karoui \cite{Jam}(we present it in a slightly weaker form to highlight the important parts).
\begin{theorem}\label{Bonami thm}
For $0 \le n \le c$ and $c > 100$ we have
\begin{equation}\label{Bonami eq}
1 - \frac{(\pi c)^n}{n!}e^{-\frac{\pi}{2}c} \le \lambda_n(c) < 1.
\end{equation}
\end{theorem}
Note that this lower bound is meaningful only if $\frac{(\pi c)^n}{n!}e^{-\frac{\pi}{2}c}\le 1$, which asymptotically means $n \le 0.58 c$, and for such values of $n$ we have an exponential lower bound $\lambda_n(c) \ge 1 - \gamma^c$ for some $\gamma < 1$. 

The goal of the present work is to push this estimate all the way to the plunge region and obtain the following result.
\begin{theorem}\label{main}
For any $\eps > 0$ there exists constant $0 < \delta = \delta(\eps) < 1$ such that for large enough $c$ we have
$$\lambda_n(c) \ge 1 - \delta^c,$$
where $n = [(1-\eps)c]$.
\end{theorem}
Note that exponential lower bound is the best we can achieve since this is the best possible bound already for the $\lambda_1(c)$.

The proof of Theorem \ref{Bonami thm} is based on applying the min-max principle to the operator $S_{T, \Omega}$ and picking a suitable subspace of $L^2(\R)$. The subspace that they chose is generated by the first $n$ Hermite functions. However, there are two reasons why they only prove the result for $\eps \ge 0.42$. The first one is that their elementary real--analytic estimates for the tails of the Hermite functions are not the strongest possible. The second one is that even with the best possible bounds on the Hermite functions we can not prove the theorem in full generality -- we must consider more general subspaces generated by the time-frequency shifts of Hermite functions and the best we can get from the Hermite functions alone is $\eps > 1-\frac{\pi}{4}=0.21$.

Our approach is to also apply the min-max principle, but translate the problem to the realm of complex analysis by means of the Bargmann transform and work with the functions in the Bargmann--Segal--Fock space. Although the projection $P_\Omega$ becomes not the most pleasant operator in this setting, since we only care about the lower bounds we can approximate it by some crude estimates which are enough to get Theorem \ref{main}. Additionally, since the time-frequency shifts of the Hermite functions are no longer orthogonal we also need to upper bound their inner products to show that they are almost orthogonal. 

It turns out that in the Bargmann--Segal--Fock space the $k$'th Hermite function lives essentially on the disk of area $k$ centred at the origin, and, if we let $T = \Omega = [-\frac{\sqrt{c}}{2}, \frac{\sqrt{c}}{2}]$, to have good bounds on the operator $S_{T, \Omega}$ we need this disk to lie within the square $T\times \Omega$. This gives us the aforementioned value of $\eps > 1 - \frac{\pi}{4}$. When we apply the time-frequency shift we shift this disk on the complex plane and the Hermite functions corresponding to different disks have small inner product if these disks do not intersect. This idea naturally leads us to the following purely geometric lemma.
\begin{lemma}\label{disk lemma}
For any $\eps > 0$ there exists a finite union of pairwise disjoint closed disks $D_1, D_2, \ldots, D_N \subset (-\frac{1}{2}, \frac{1}{2})\times (-\frac{1}{2}, \frac{1}{2})$ such that the area of their union is at least $1-\eps$.
\end{lemma}
Note that since the disks we consider are closed and the square we consider is open, there exists $\gamma = \gamma(\eps) > 0$ such that all disks are at least $\gamma$ away from the boundaries of the square and from each other. 

Harder version of this lemma was given as a problem at the first USSR mathematical olympiad for students \cite{USSR}. To keep the article self-contained we  present its proof at the end of the text.

It is worth noting that our interest in the lower bounds for the eigenvalues of the time-frequency operator came from our study of the Fourier interpolation formulas \cite{Kul}, such as the Radchenko--Viazovska formula \cite{Rad}. In particular, Theorem \ref{main} allows us to prove the main result of \cite{Kul}, although with a weaker error term $o(4WT)$ instead of $O(\log^2(4WT))$. We also find it interesting that the formula in \cite{Rad} came from the celebrated work of Viazovska on sphere packings \cite{Via} and our proof is based on a circle packing of the square from Lemma \ref{disk lemma}, although it is crucial that we allow disks of varying radii. 

The structure of this paper is as follows. In Section 2 we recall the definitions and basic properties of the Bargmann transform and the Bargmann--Segal--Fock space. In Section 3 we construct an almost orthogonal system of functions with strong time-frequency localization properties and in Section 4 using simple functional analysis we prove Theorem \ref{main}. Finally, in the last section we prove Lemma \ref{disk lemma}.
\section{Bargmann transform}
Let $f$ be a function in $L^2(\R)$. For $z\in \Cm$ we define the Bargmann transform of $f$ at $z$ as 
\begin{equation}\label{bargmann}
\B f(z) = 2^{1/4} \int_\R f(t) e^{2\pi t z - \pi t^2 - \frac{\pi}{2}z^2}dt.
\end{equation}
The function $\B f$ turns out to be an entire function belonging to the Bargmann--Segal--Fock space $\F$ of entire functions for which the norm
$$||F||_{\F}^2 = \int_\Cm |F(z)|^2 e^{-\pi |z|^2}dz$$
is finite. Moreover, $||f||_{L^2(\R)} = ||\B f||_{\F}$ and the Bargmann transform is a bijection between $L^2(\R)$ and $\F$. For the proofs of these facts see \cite[Section 3.4]{Gro}.

Since the space $L^2(\R)$ has a rich group of isometries, coming from translations and modulations, they should have a counterpart on the Bargmann transform side. Indeed, for $w\in \Cm$ and $F\in \F$ consider
$$T_w F(z) = F(z-w)e^{\pi z\bar{w}-\frac{\pi}{2}|w|^2}.$$

Clearly, $F_w$ is still an entire function and direct computation shows that $||T_wF||_\F = ||F||_\F$, in particular if $F\in \F$ then $T_wF\in \F$.

As usual for the spaces of analytic functions, point evaluations are continuous in $\F$. Specifically, for all $F\in\F$ and $z\in \Cm$ we have $|F(z)| \le e^{\frac{\pi}{2}|z|^2}||F||_\F $, which can be verified from \eqref{bargmann} and the Cauchy--Schwarz inequality. 

The space $\F$ contains an orthonormal basis $\{ \sqrt{\frac{\pi^k}{k!}}z^k\}_{k\in \N_0}$.  While it is not technically necessary for our proof, it is important to mention that these functions are exactly the Bargmann transform of the Hermite functions $h_k(x)$. Thus, this basis is the Bargmann transform of the basis used by Bonami, Jaming and Karoui. We will consider more generally $T_w$ applied to this basis for various values of $w$.

Last fact that we need is that the Fourier transform corresponds to rotation by $90$ degrees in the complex plane, that is $(\B \F f)(z) = \B f(iz)$. This can be either verified using the aforementioned fact about the Hermite functions, or derived directly from \eqref{bargmann} by using the fact that the Fourier transform is an isometry and the Fourier transform of a Gaussian is a Gaussian.
\section{Construction of the system}
Since eigenvalues $\lambda_n(I, J)$ for the intervals $I, J$ depend only on the product $|I| |J|$, we will assume without loss of generality  that $I = J = [-\frac{\sqrt{c}}{2}, \frac{\sqrt{c}}{2}]$. Let us fix an $\eps > 0$ in the Theorem \ref{main} and consider the set of disks $D_1, D_2, \ldots , D_N$ from Lemma \ref{disk lemma} corresponding to this $\eps$. By $w_m$ and $r_m$ we denote the center and radius of $D_m$, respectively.

We consider the following set of functions from $\F$
\begin{equation}\label{bargmann set}
\mathcal{U} = \{ T_{\sqrt{c}w_m} \frac{z^k}{\sqrt{k!}} \mid 1 \le m \le N, 0 \le k \le c \pi r_m^2\}
\end{equation}
and put $\mathcal{V} = \B^{-1} \mathcal{U}$ (the set $\mathcal{V}$ consists of some time-frequency shifts of Hermite functions). Note that all functions in $\mathcal{V}$ have $L^2(\R)$ norm equal to $1$ and that $\mathcal{V}$ contains at least $(1-\eps)c$ elements. The goal of this section is to show that the functions from $\mathcal{V}$ are almost orthogonal and have very strong concentration on the interval $I$ and Fourier concentration on the interval $J$. Specifically, we will prove the following proposition.
\begin{proposition}\label{proposition}
There exists $\alpha = \alpha(\eps) < 1$ such that for big enough $c$ the following conditions hold.
\begin{enumerate}
\item For all $f, g\in \mathcal{V}, f\neq g$ we have $|\langle f, g\rangle| \le \alpha^c$.

\item For all $f\in \mathcal{V}$ we have $||(\Id - P_I)f||\le \alpha^c$ and $||(\Id - Q_J)f|| \le \alpha^c$.
\end{enumerate}
\end{proposition}
\begin{proof}
We begin with proving $(i)$. Since Bargmann transform is an isometry, it is enough to bound $|\langle \B f, \B g\rangle|$. Set $F = \B f, G = \B g$. If $F$ and $G$ correspond to the same disk $D_m$ then they are orthogonal since the monomials $z^k$ are pairwise orthogonal and $T_{\sqrt{c}w_m}$ is an isometry of the space $\F$. Thus, we can assume that $F$ corresponds to the disk $D_n$ and $G$ corresponds to the disk $D_m$, $m\neq n$. Put $v_n = \sqrt{c}w_n, v_m = \sqrt{c}w_m$ and assume that $F = T_{v_n} \frac{z^k}{\sqrt{k!}}, G = T_{v_m}\frac{z^l}{\sqrt{l!}}$.  We are going to estimate $|\langle F, G\rangle|$ directly. We have
\begin{align*}
|\langle F, G\rangle| \le\int_\Cm |F(z)||G(z)|e^{-\pi |z|^2}dz \le\\ \int_{|z-v_n| > \sqrt{c}(r_n + \frac{\gamma}{2})}  |F(z)||G(z)|e^{-\pi |z|^2}dz+ \int_{|z-v_l| > \sqrt{c}(r_m + \frac{\gamma}{2})}  |F(z)||G(z)|e^{-\pi |z|^2}dz,
\end{align*}
where in the second inequality we used that $D_n$ and $D_m$ are at least $\gamma$ apart. We will estimate only the first of these two integrals since the estimate for the other one is similar.

Since $||G||_\F = 1$, we have $|G(z)|e^{-\frac{\pi}{2}|z|^2} \le 1$. Therefore 
$$\int_{|z-v_n| > \sqrt{c}(r_n + \frac{\gamma}{2})}  |F(z)||G(z)|e^{-\pi |z|^2}dz \le \int_{|z-v_n| > \sqrt{c}(r_n + \frac{\gamma}{2})}  |F(z)|e^{-\frac{\pi}{2} |z|^2}dz.$$
We will take this integral in polar coordinates with $|z-v_n| = r, \arg (z-v_n) = \theta$. This way we get
$$2\pi\sqrt{\frac{\pi^k}{k!}}\int_{\sqrt{c}(r_n + \frac{\gamma}{2})}^\infty r^{k+1}e^{-\frac{\pi}{2}r^2}dr.$$
Note that to justify this formula we can either do a direct computation or first apply $T_{v_n}^{-1}$ to the whole thing since it is an isometry of $\F$. Next, we do the change of variables $s = r^2$ and get
$$\pi \sqrt{\frac{\pi^k}{k!}}\int_{c(r_n+\frac{\gamma}{2})^2}^\infty s^{\frac{k}{2}}e^{-\frac{\pi}{2}s}ds.$$
Observe that the expression in the integral is decreasing in $s$. Indeed, differentiating it we can see that it is decreasing if $k < \pi s$ and since $k \le  c \pi r_n^2$, we can see that it is decreasing.  Therefore, the integral is at most
$$\sum_{p = 0}^\infty \left(p + c\left(r_n+\frac{\gamma}{2}\right)^2\right)^{\frac{k}{2}} e^{-\frac{\pi}{2}(p+c(r_n+\frac{\gamma}{2})^2)}.$$

Let us estimate the ratio of two consecutive elements of this sum, setting $p + c(r_n+\frac{\gamma}{2})^2 = s$ for brevity. We have
$$\left(\frac{s+1}{s}\right)^{\frac{k}{2}}e^{-\frac{\pi}{2}} \le e^{\frac{k}{2s} - \frac{\pi}{2}} \le e^{\frac{c\pi r_n^2}{2c(r_n+\frac{\gamma}{2})^2}- \frac{\pi}{2}} = e^{\frac{\pi}{2}\left(\frac{1}{(1+\frac{\gamma}{2r_n})^2} - 1\right)},$$
where in the first inequality we used the well-known fact that $(1+\frac{1}{s})^s \le e$. Therefore, the ratio of any two consecutive terms is at most some $\beta = \beta(\eps) < 1$. Hence, the whole sum is at most the first term multiplied by some constant depending only on $\eps$. Thus, the integral is at most
$$C_\eps\pi \sqrt{\frac{\pi^k}{k!}}\left(c\left(r_n+\frac{\gamma}{2}\right)^2\right)^{\frac{k}{2}} e^{-\frac{\pi}{2}c(r_n+\frac{\gamma}{2})^2}.$$

The last ingredient that we need is the classical inequality $k! \ge \left(\frac{k}{e}\right)^k$. Collecting everything, we get the upper bound
$$C_\eps \pi \left(\frac{ce\pi\left(r_n+\frac{\gamma}{2}\right)^2}{k}\right)^{\frac{k}{2}}  e^{-\frac{\pi}{2}c(r_n+\frac{\gamma}{2})^2}.$$

Differentiating this quantity with respect to $k$ we can see that it is increasing for $k < c\pi \left (r_n + \frac{\gamma}{2}\right)^2$, in particular we can without loss of generality assume that $k = c\pi r_n^2$. Substituting this value we get
$$C_\eps \pi  \left(e\left(1+\frac{\gamma}{2r_n}\right)^2\right)^{\frac{c\pi r_n^2}{2}}e^{-\frac{\pi}{2}c(r_n+\frac{\gamma}{2})^2} = C_\eps \pi \left(e\left(1+\frac{\gamma}{2r_n}\right)^2e^{-(1+\frac{\gamma}{2r_n})^2} \right)^{\frac{c\pi r_n^2}{2}}.$$

Let $u = \left(1+\frac{\gamma}{2r_n}\right)^2 - 1 > 0$. Then the quantity in the brackets is $(1+u)e^{-u} = \nu < 1$, therefore the whole expression is 
$$C_\eps \pi \nu^{c\frac{\pi r_n^2}{2}},$$
which is exponentially small in $c$, as required. Doing the same for the second integral we get $C_\eps' \pi \nu'^{c\frac{\pi r_m^2}{2}}$. Let $\nu_0$ be the maximum of all $\nu$ over all pairs $(n, m)$, $c_\eps$ be the maximum of $C_\eps$ over all pairs $(n, m)$ and $r$ be the minimum of $r_n$. Then we get the upper bound
$$2c_\eps \pi \nu_0^{c\frac{\pi r^2}{2}}.$$

Taking $\nu_0^{\frac{\pi r^2}{2}} < \alpha < 1$ and taking $c$ big enough we can have this quantity less than $\alpha^c$, as required.

Now, we turn to proving $(ii)$. We will only bound $||(\Id - P_I)f||$ since the Fourier transform corresponds to the 90 degrees rotation of the complex plane, and when we rotate the square with center at the origin we get back the same square (in fact, with slight tweaking to the proof of Lemma \ref{disk lemma} we can achieve that set $\mathcal{V}$ is invariant under the Fourier transform). We have 
$$||(\Id - P_I)f||^2 = \int_{-\infty}^{-\frac{\sqrt{c}}{2}} |f(t)|^2 dy + \int_{\frac{\sqrt{c}}{2}}^{\infty} |f(t)|^2 dy.$$

We will only bound the second term, the other one being similar. To estimate it, we use duality:
$$\int_{\frac{\sqrt{c}}{2}}^{\infty} |f(t)|^2 dy = \sup_{\supp g\subset [\frac{\sqrt{c}}{2}, +\infty), ||g||_2 = 1} |\langle f, g\rangle|^2.$$

Since Bargmann transform is an isometry, this inner product is equal to $\langle \B f, \B g\rangle$. Recall that $\B f$ is just a $T_w$ shift of some normalized monomial. Specifically, let $f$ correspond to the disk $D_n$, with center $w_n$ and radius $r_n$, so that $\B f = F = T_{v_n} \frac{z^k}{\sqrt{k!}}$, where $v_n = \sqrt{c}w_n$. We want to obtain an upper bound for
$$\int_\Cm |F(z)| |G(z)| e^{-\pi |z|^2}dz,$$
where $G = \B g$. We will again split this integral into two integrals
$$\int_{|z - v_n| \ge\sqrt{c}(r_n + \frac{\gamma}{2})}|F(z)| |G(z)| e^{-\pi |z|^2}dz+\int_{|z - v_n| < \sqrt{c}(r_n + \frac{\gamma}{2})}|F(z)| |G(z)| e^{-\pi |z|^2}dz.$$

For the first integral we bound $|G(z)|e^{-\frac{\pi}{2}|z|^2} \le 1$ and proceed exactly like in the part $(i)$. For the second integral, we in turn use an estimate $|F(z)|e^{-\frac{\pi}{2}|z|^2} \le 1$ and so we need to obtain a stronger pointwise bound for $|G(z)|$ when $|z - v_n| \le \sqrt{c}(r_n + \frac{\gamma}{2})$. We will do this by applying the Cauchy--Schwarz inequality to the definition of the Bargmann transform \eqref{bargmann}. Note that this will be the only place where we use an explicit formula for the Bargmann transform. We have
$$e^{-\frac{\pi}{2}|z|^2}\B g(z) = e^{-\frac{\pi}{2}|z|^2} 2^{1/4}\int_{\frac{\sqrt{c}}{2}}^\infty  g(t) e^{2\pi t z - \pi t^2 - \frac{\pi}{2}z^2}dt.$$

Applying the Cauchy--Schwarz inequality we get

$$\left|e^{-\frac{\pi}{2}|z|^2}\B g(z)\right|^2 \le e^{-\pi|z|^2} 2^{1/2} \int_{\frac{\sqrt{c}}{2}}^\infty e^{\re(4\pi t z - 2\pi t^2 -\pi z^2)}dt.$$

This quantity can be seen to be the tail of the Gaussian distribution. This can be checked directly, but to guess this we can use the following heuristic: if the lower limit was $-\infty$ then this would have been just $1$ from the pointwise bound in the Bargmann--Segal--Fock space, and by applying the shift $T_z^{-1}$ we will get the same result as if we were integrating from $\frac{\sqrt{c}}{2} - \re z$. So, we have the bound
$$\left|e^{-\frac{\pi}{2}|z|^2}\B g(z)\right|^2 \le \sqrt{2}\int_{\frac{\sqrt{c}}{2} - \re z}^\infty e^{-2\pi t^2}dt.$$

Observe that if $|z - v_n| \le \sqrt{c}(r_n + \frac{\gamma}{2})$ then $\re z \le \frac{\sqrt{c}}{2}(1 - \frac{\gamma}{2})$, because disk $D_n$ is at least $\gamma$ away from the boundary of the square $(-\frac{1}{2}, \frac{1}{2})\times (-\frac{1}{2}, \frac{1}{2})$. Therefore, for such $z$ we have
$$\left|e^{-\frac{\pi}{2}|z|^2}\B g(z)\right|^2 \le \sqrt{2}\int_{\frac{\gamma\sqrt{c}}{2}}^\infty e^{-2\pi t^2}dt.$$

As in the part $(i)$, we observe that the function $e^{-2\pi t^2}$ is decreasing and $e^{-2\pi(t+1)^2 + 2\pi t^2} \le e^{-2\pi} < 1$ for $t \ge 0$, therefore the integral is at most $C_\eps$ multiplied by $e^{-2\pi \left(\frac{\gamma\sqrt{c}}{2}\right)^2}$ for some $C_\eps > 0$. That is, we have 
$$\left|e^{-\frac{\pi}{2}|z|^2}\B g(z)\right| \le \sqrt{C_\eps\sqrt{2}} e^{-\frac{\pi \gamma^2}{4}c}.$$

When we integrate this bound over the disk $|z - v_n| \le \sqrt{c}(r_n + \frac{\gamma}{2})$ we multiply by the area of this disk. Since it is contained in the square  $(-\frac{\sqrt{c}}{2}, \frac{\sqrt{c}}{2})\times (-\frac{\sqrt{c}}{2}, \frac{\sqrt{c}}{2})$, its area is at most $c$. Thus, the integral is at most
$$c \sqrt{C_\eps\sqrt{2}} e^{-\frac{\pi \gamma^2}{4}c}.$$

This is exponentially decreasing in $c$, thus we proved that for some $\alpha < 1$ we have $|\langle F, G\rangle| \le \alpha^c$ if $c$ is big enough, as required.
\end{proof}
\section{Proof of Theorem \ref{main}}
Let us take a subset $\mathcal{V}_0 \subset\mathcal{V}$ of size $n = [(1-\eps)c]$ and let $V \subset L^2(\R)$ be the subspace generated by $\mathcal{V}_0$. We will apply min-max principle to $V$ to deduce the lower bound for $\lambda_n(c)$ (in particular, we will show that functions in $\mathcal{V}_0$ are linearly independent so that the dimension of $V$ is indeed $[(1-\eps)c]$).

We start with estimating $P_IQ_JP_If$ for $f\in \mathcal{V}$. We have 
$$||(\Id-Q_JP_I)f|| = ||(\Id-Q_J)f+Q_J(\Id-P_I)f|| \le  ||(\Id-Q_J)f||+||(\Id-P_I)f|| \le 2\alpha^c,$$
where we used $||Q_J|| \le 1$ since $Q_J$ is a projection. Doing this one more time we get
$$||(\Id-P_IQ_JP_I)f|| = ||P_I(\Id-Q_JP_I)f+(\Id-P_I)f|| \le ||(\Id-Q_JP_I)f|| + ||(\Id-P_I)f|| \le 3\alpha^c.$$

Let $\mathcal{V}_0 = \{f_1, f_2, \ldots , f_n\}$ and consider $f(t) = \sum_{k = 1}^n a_kf_k(t)$ where not all $a_k$ are equal to zero. We are interested in the quantity $\frac{||P_IQ_JP_If||}{||f||}$. We have
$$||f||^2 = \sum_{k = 1}^n |a_k|^2 + \sum_{1 \le k, l \le n, k\neq l} a_k \bar{a_l}\langle f_k, f_l\rangle.$$

From the Proposition \ref{proposition} we know that $|\langle f_k, f_l\rangle| \le \alpha^c$. Therefore, we can crudely bound
$$||f||^2 \ge \sum_{k = 1}^n |a_k|^2 - \left(\sum_{k = 1}^n |a_k|\right)^2\alpha^c.$$

Now, we can show that $\mathcal{V}_0$ is linearly independent. By the inequality between the arithmetic mean and the quadratic mean we have $\sum_{k = 1}^n a_k^2 \ge \frac{1}{n} \left(\sum_{k = 1}^n |a_k|\right)^2$, therefore, if $\alpha^c < \frac{1}{2c}$, we have 
\begin{equation}\label{crude}
||f||^2 \ge (1 - n\alpha^c) \sum_{k = 1}^n |a_k|^2\ge \frac{1}{2}\sum_{k = 1}^n |a_k|^2,
\end{equation}
in particular it is non-zero (here we used a crude bound $n \le c$).

Similarly, we can estimate $||P_IQ_JP_If||$. We have
$$||(\Id-P_IQ_JP_I)f|| = ||\sum_{k=1}^n (\Id - P_IQ_JP_I)a_k f_k|| \le \sum_{k = 1}^n |a_k| ||(\Id - P_IQ_JP_I)f_k|| \le 3\alpha^c \sum_{k = 1}^n |a_k|.$$
By the triangle inequality this implies
$$||f|| -3\alpha^c \sum_{k = 1}^n |a_k| \le  ||P_IQ_JP_If|| .$$

Dividing this by $||f||$ we get
$$\frac{||P_IQ_JP_If||}{||f||} \ge 1 - \frac{3\alpha^c \sum_{k = 1}^n |a_k|}{||f||}.$$

 Using the estimate \eqref{crude} and  the inequality between the arithmetic mean and the quadratic mean  this is at least $1 - 3\sqrt{2c}\alpha^c$. Choosing $\alpha < \delta < 1$ and $c$ big enough this is at least $ 1 -\delta^c$. By the min-max principle this means that $\lambda_n(c) \ge 1 -\delta^c$, as required.
\section{Proof of Lemma \ref{disk lemma}}
We will prove by induction on $n$ that there exists a finite set of disjoint closed disks in $R = (-\frac{1}{2}, \frac{1}{2})\times (-\frac{1}{2}, \frac{1}{2})$ with total measure at least $1 - 2^{-n}$. For $n = 1$ we can take disk with center at the origin and radius $0.49$, its area is $\pi 0.49^2 > \frac{1}{2}$.

Assume that we already constructed disks $D_1, \ldots , D_m$ with total measure at least $1-2^{-n}$. We want to add some disks so that the total measure becomes at least $1-2^{-n-1}$.

Let us pick a very large integer $N$ to be determined later and slice $R$ into the squares of side length $\frac{1}{N}$. We will split the squares into three groups: squares in the group $A$ are those which lie entirely inside one of the disks, squares in the group $B$ are those which do not intersect any of the disks, and squares in the group $C$ are those which intersect the boundary of one of the disks.

We claim that there exists a constant $c$, independent of $N$ (but depending on the already chosen disks $D_1, \ldots , D_m$) such that the number of squares in $C$ is at most $c N$. Indeed, let us consider any square $H$ from the group $C$ and assume that it intersects the boundary of $D_l$. Since the diameter of $H$ is $\frac{\sqrt{2}}{N}$, square $H$ lies entirely in the $\frac{\sqrt{2}}{N}$--vicinity of the circumferences of $D_l$. Therefore, the total measure of disks in $C$ is not greater than the total measure of the union of $\frac{\sqrt{2}}{N}$--vicinities of the circumferences of the disks $D_1, \ldots , D_m$. Let the radius of $D_l$ be $r_l$. Then the measure of its $\frac{\sqrt{2}}{N}$--vicinity is 
$$\pi\left(r_l + \frac{\sqrt{2}}{N}\right)^2 - \pi\left(r_l - \frac{\sqrt{2}}{N}\right)^2 = \frac{\pi4\sqrt{2}r_l}{N},$$
assuming that $N$ is chosen so big that  $\frac{\sqrt{2}}{N}$ is smaller than the radius of every disk $D_l$. Therefore, the measure of the unions of these vicinities is at most
$$\frac{\pi4\sqrt{2}\sum_{l = 1}^m r_l}{N}.$$

Since each square $H$ from our decomposition has area $\frac{1}{N^2}$, there can be at most $\pi4\sqrt{2}\sum_{l = 1}^m r_l N$ squares in the set $C$, so $c=\pi4\sqrt{2}\sum_{l = 1}^m r_l$ works.

We have $|A| +|B| + |C| = N^2$. Let $\mu$ be the measure of the union of the already chosen disks $D_1, \ldots , D_m$. By the induction hypothesis $\mu \ge 1 - 2^{-n}$. For each square in $B$ we will consider the disk with the center at the center of $B$ and radius $\frac{0.49}{N}$. Clearly, these disks are pairwise disjoint, they are disjoint from the already chosen disks $D_1, \ldots , D_m$ and they lie inside of $R$ (since $0.49 < 0.50$, they don't even touch each other or the boundary of $R$, as required). Their total measure is $\frac{|B|0.49^2 \pi}{N^2} \ge \frac{3|B|}{4N^2}$. Adding them to our set, the total measure we get is at least $\frac{3|B|}{4N^2} + \mu$. It remains to show that this is at least $1-2^{-n-1}$.

Since all squares from $A$ lie inside of the disks $D_1, \ldots , D_m$, we have $\frac{|A|}{N^2} \le \mu$. Subtracting this inequality from $1$, we get
$$ 1 - \mu \le \frac{|B| + |C|}{N^2}.$$

Subtracting further $\frac{3|B|}{4N^2}$ from both sides, we obtain
\begin{equation}\label{lemma eq}
1 - \mu - \frac{3|B|}{4N^2} \le \frac{|B| + 4|C|}{4N^2}.
\end{equation}
Since all the squares from $B$ do not intersect any of the disks $D_1, \ldots , D_m$, their total measure is at most $1-\mu$, that is $\frac{|B|}{N^2} \le 1 - \mu \le 2^{-n}$. Plugging this into \eqref{lemma eq} and recalling that $|C| \le cN$ we get
$$1 - \mu - \frac{3|B|}{4N^2} \le 2^{-n-2} + \frac{c}{N}.$$

Choosing $N$ so that $\frac{c}{N} < 2^{-n-2}$ we get the desired result.
\subsection*{Acknowledgments} I would like to thank Fabio Nicola, Kristian Seip and Mikhail Sodin for helpful discussions. This work was supported by BSF Grant 2020019, ISF Grant 1288/21, and by The Raymond and Beverly Sackler Post-Doctoral Scholarship.  



\end{document}